\documentclass[a4paper,12pt]{amsart}

\usepackage{amsmath, amsthm, amssymb, amsfonts}
\usepackage[normalem]{ulem}
\usepackage{hyperref}

\usepackage{verbatim} 
\usepackage{longtable} 
\usepackage{enumitem}
\usepackage{relsize}

\theoremstyle{plain}
\newtheorem{thm}{Theorem}
\newtheorem{cor}[thm]{Corollary}
\newtheorem{lemma}[thm]{Lemma}

\theoremstyle{definition}

\theoremstyle{remark}

\newcommand{\BC}{{\mathbb{C}}}

\newcommand{\BE}{{\mathbb{E}}}

\newcommand{\BP}{{\mathbb{P}}}
\newcommand{\BQ}{{\mathbb{Q}}}
\newcommand{\BR}{{\mathbb{R}}}

\newcommand{\BZ}{{\mathbb{Z}}}

\newcommand{\CM}{{\mathcal M}}
\newcommand{\CN}{{\mathcal N}}
\newcommand{\CO}{{\mathcal O}}

\newcommand{\CX}{{\mathcal X}}

\DeclareMathOperator{\Sym}{Sym}

\newcommand{\pr}{\mathop{\rm pr}\nolimits}
\newcommand{\Bl}{\mathop{\rm Bl}\nolimits}
\newcommand{\Pic}{\mathop{\rm Pic}\nolimits}

\newcommand{\Mbar}{{\overline M}}

\newcommand{\End}{{\rm End}}
\newcommand{\Hdg}{{\rm Hdg}}
\newcommand{\Image}{{\rm Im}}
\newcommand{\SO}{{\rm SO}}
\newcommand{\Orth}{{\rm O}}

\newcommand{\ev}{\mathop{\rm ev}\nolimits}

\newcommand{\one}{\mathsf{1}}
\newcommand{\pt}{{\mathsf{p}}}
\newcommand{\sfa}{{\mathsf{a}}}
\newcommand{\sfb}{{\mathsf{b}}}

\newcommand{\GW}[2] {\Omega^{#1}_{#2} }

\usepackage{tikz}
\usepackage{tikz-cd}
\tikzcdset{row sep/normal=0.2cm}
\AtBeginDocument{%
	\def\MR#1{}
}
\begin{document}
\baselineskip=16pt
\parskip=5pt

\title[Gromov--Witten classes of K3 surfaces]
{Gromov--Witten classes of K3 surfaces}
\author {Tim-Henrik Buelles}
\address {ETH Z\"urich, Department of Mathematics}
\email {buelles@math.ethz.ch}
\date{\today}
\begin{abstract} We study the cycle-valued reduced Gromov--Witten theory of a nonsingular projective K3 surface. For primitive curve classes, we prove that the correspondence induced by the reduced virtual fundamental class respects the tautological rings. Our proof uses monodromy over the moduli space of K3 surfaces, degeneration formulae and virtual localization. As a consequence of the monodromy argument, we verify an invariance property for Gromov--Witten invariants of K3 surfaces in primitive curve class conjectured by Oberdieck--Pandharipande.
\end{abstract}

\maketitle
	
	\setcounter{section}{-1}
	
\section{Introduction}
\subsection{Gromov--Witten classes}

Let $\Mbar_{g,n}(X,\beta)$ be the moduli space of stable maps from connected genus $g$ curves with $n$ marked points to a target variety $X$, representing the curve class $\beta\in H_2(X,\BZ)$.\footnote{We will always assume $\beta$ effective and $2g-2+n >0$.} The moduli space comes with natural projection and evaluation maps
\[\begin{tikzcd}
	&\Mbar_{g,n}(X,\beta) \ar[dl,"\pi",swap] \ar[dr,"\ev"] & \\
	\Mbar_{g,n} && X^n.
\end{tikzcd}\] 
The action of the virtual fundamental class defines two series of cycles:
\begin{align*} Z^X_{g,n,\beta}(\alpha) &= \ev_*\left(\pi^* \alpha\cap[\Mbar_{g,n}(X,\beta)]^{vir} \right)\in H^*(X^n)\, , &\alpha&\in H^*(\Mbar_{g,n})\, ,
\\ \intertext{and, considering the action in the opposite direction,} \Omega^X_{g,n,\beta}(v)&=\pi_*\left(\ev^*v \cap [\Mbar_{g,n}(X,\beta)]^{vir}\right)\in H^*(\Mbar_{g,n})\, , &v&\in H^*(X^n)\, . 
\end{align*}
The latter are sometimes called \textit{Gromov--Witten classes}. Classical Gromov--Witten invariants are obtained by integrating against tautological classes on $\Mbar_{g,n}$. Our goal here is to study these classes for a nonsingular projective K3 surface $S$. Since ordinary Gromov--Witten theory of a K3 surface vanishes for non-zero curve classses, we replace the virtual class by the \textit{reduced virtual class}\footnote{We refer to \cite{KT14} for a modern treatment of the reduced obstruction theory.}
\[[\Mbar_{g,n}(S,\beta)]^{red}\, . \]
To locate the Gromov--Witten classes in cohomology, we consider the tautological subrings~\cite{FP05}
\[ R^*(\Mbar_{g,n}) \subset H^*(\Mbar_{g,n})\, , \] 
spanned by push-forwards of products of $\psi$ and $\kappa$ classes on boundary strata. On the K3 side, we consider the subrings~\cite{Vo08}
\[R^*(S^n)\subset H^*(S^n) \]
generated by the diagonals and the pullbacks of $\beta$ under the projection maps.\footnote{The dependence on $\beta$ is suppressed in the notation. Also, we view $\beta$ as a cohomology class under the natural isomorphism  $H_2(S,\BZ)\stackrel{\sim}{=}H^2(S,\BZ)$.} 

We prove that the reduced virtual class yields a correspondence between tautological rings.
\begin{thm}\label{thm:1}
	Let $\beta\in H_2(S,\BZ)$ be primitive and $\alpha\in R^*(\Mbar_{g,n})$, then
	\[Z^S_{g,n,\beta}(\alpha) \in R^*(S^n)\, . \]
\end{thm}
\begin{thm} \label{thm:2}
	Let $\beta\in H_2(S,\BZ)$ be primitive and $v\in R^*(S^n)$, then
	\[ \Omega^S_{g,n,\beta}(v)\in R^*(\Mbar_{g,n})\, . \]
\end{thm}
We expect both results to hold for arbitrary curve classes. A careful study of the degeneration formula for the reduced virtual class is necessary here. Also, the above correpondence is conjectured to hold for algebraic cycles~\cite[Conj.\ 1 and 2]{PY16}. Our arguments however do not apply to Chow groups.

The system of Gromov--Witten classes satisfies certain compatibilities captured by the notion of a Cohomological Field theory~\cite{KM94, Pa17}. Although examples of non-tautological CohFT's exist~\cite{PZ18}, it is not known whether every CohFT that arises from the Gromov--Witten theory of a variety $X$ is in fact tautological.\footnote{This question already appears in~\cite{FP05}.}

Let us point out that Theorem~\ref{thm:2} was previously known for curves~\cite{FP05,Ja17} and varieties equipped with a strong torus action (with finitely many $0$- and $1$-dimensional orbits)~\cite{GP99}. Our proof uses these known cases.

\subsection{Invariance property}
As a consequence of Theorem~\ref{thm:1} we verify an invariance property for Gromov--Witten invariants of K3 surfaces conjectured by Oberdieck--Pandharipande~\cite{OP16}. Let $S$ and $\widetilde{S}$ be two K3 surfaces, and let
\[\varphi\colon\Big(H^2(S,\BR)\, ,\langle,\rangle\Big) 
\rightarrow \Big(H^2(\widetilde{S}, \BR)\, ,\langle,\rangle \Big) \]
be a {\em real isometry} sending
an effective curve class $\beta\in H^2(S,\BZ)$ to an
effective curve class $\widetilde{\beta}\in H^2(\widetilde{S},\BZ)$,
\[\varphi(\beta)= \widetilde{\beta}\, .\]
We extend $\varphi$ to $H^*(S,\BR)$ by
\[\varphi(\one)=\one\, , \ \ \ \ \varphi(\pt)=\pt\, , \]
where $\one,\pt\in H^*(S,\BR)$ are the identity and the point class respectively.

Denote the (reduced) descendent Gromov--Witten invariants of $S$ by \[ \left\langle \prod_{i=1}^n \tau_{a_i}(w_i) \right\rangle_{g,\beta}^S=
\int_{[\overline{M}_{g,n}(S,\beta)]^{red}} \prod_{i=1}^n \psi_i^{a_i}
\cup \text{ev}_i^*(w_i)\ , \ \ \ w_i\in H^*(S,\mathbb{Q})\ .\]
Then the following invariance property holds:
\begin{cor}\label{cor:Inv}
Let $\beta\in H^2(S,\BZ)$ and $\widetilde{\beta}\in H^2(\widetilde{S},\BZ)$ be primitive curve classes, then
\[
\left\langle \prod_{i=1}^n \tau_{a_i}(w_i) \right\rangle_{g,\beta}^S
= 
\left\langle \prod_{i=1}^n \tau_{a_i}(\varphi(w_i)) 
\right\rangle_{g,\widetilde{\beta}}^{\widetilde{S}}.
\]
\end{cor}
The equality is expected to hold for all curve classes $\beta$ and $\widetilde{\beta}$ of the same divisibility. The missing ingredient for imprimitive curve classes is a degeneration formula for the reduced virtual class and a suitable genus reduction.\footnote{See the comment in the proof of Lemma~\ref{lemma:Vanishing}.} Once established, our proof may be extended to yield the full invariance property.

Corollary~\ref{cor:Inv} now fully justifies the formulation of the Multiple Cover Conjecture for K3 surfaces~\cite[Conj.\ C2]{OP16}.
\subsection{Outline}
Section~\ref{sec:Monodromy} contains a recollection of the key facts about monodromy over the moduli space of K3 surfaces and Hodge groups. We then explain how to use these facts to prove Theorem~\ref{thm:1}. The key technical ingredient is Lemma~\ref{lemma:Vanishing}. The proof of Theorem~\ref{thm:2} is carried out in the last part of the paper and is logically independent of the preceding sections. We use special degenerations for elliptically fibered K3 surfaces~\cite{MPT10}
\[ S \rightsquigarrow R\cup_E R \]
breaking the K3 surface into the union of two rational elliptic surfaces intersecting along an elliptic curve. Our proof then proceeds by degeneration formulae and virtual localization to reduce to the case of an elliptic curve.
\subsection{Acknowledgements} This paper has greatly benefited from discussions with Georg Oberdieck, Rahul Pandharipande, Qizheng Yin. Related discussions with Daniel Huybrechts, Andrey Soldatenkov about moduli of K3 surfaces are gratefully acknowledged. 

This project has received funding from the European Research Council (ERC) under the European Union’s
Horizon 2020 research and innovation programme (grant agreement No 786580).

\section{Monodromy}\label{sec:Monodromy}
 
\subsection{Non-positive case}
Let $S$ be a nonsingular projective K3 surface and $\beta\in H_2(S,\BZ)$ a curve class. If $\beta=0$, then
\[ \Mbar_{g,n}(S,0) \cong \Mbar_{g,n} \times S \] 
and the virtual class is given by
\[[\Mbar_{g,n}(S,0)]^{vir} = \begin{cases} [\Mbar_{0,n} \times S], &g=0 \\
\pr_2^*c_2(S) \cap [\Mbar_{1,n} \times S], & g=1 \\
0, & g\geq 2.
\end{cases} \]
Theorems~\ref{thm:1} and \ref{thm:2} evidently hold true in this case. We will henceforth assume that $\beta$ is a non-zero curve class. However,  contracted components will play a role for our vanishing result Lemma~\ref{lemma:Vanishing}.

As explained in \cite[Prop.\ 3]{PY16}, if 
\[ \langle \beta,\beta \rangle < -2\] one can deform $S$ such that $\beta$ is not generically effective. The corresponding moduli space is thus empty and Theorems~\ref{thm:1} and \ref{thm:2} hold by deformation invariance of the reduced virtual class.

An effective class $\beta\in H_2(S,\BZ)$ satisfying
\[ \langle \beta,\beta \rangle = -2\]
is always represented by a smooth rational curve $C\subset S$. Since $C\cong \BP^1$ is rigid, the evaluation maps $\ev_i$ factor through $C$, i.e.\
\[ \Mbar_{g,n}(S,\beta)\stackrel{\ev}{\to} C\times\ldots\times C\subset S\times\ldots\times S\, . \] 
Note that in this case, restriction induces a surjection
\[R^*(S^n)\to H^*(C\times\ldots\times C) \]
and therefore Theorem~\ref{thm:1} holds.

\subsection{Positive case}
For $\langle \beta,\beta \rangle \geq0$ we will use the `big monodromy' over the moduli space of K3 surfaces. More precisely, given $(S,\beta)$, there exists a family of K3 surfaces with central fiber $S$
\[\CX \to B \]
such that $\beta$ is a $(1,1)$ class in each fiber and the image of the monodromy representation
\[\rho\colon\pi_1(B,0)\to \Orth(H^2(S,\BZ))\] is a finite index subgroup of the subgroup of $\Orth(H^2(S,\BZ))$ fixing $\beta$. See~\cite[Sec.\ 6.4.3]{Hu16} for details. We write 
\[T= \langle \beta \rangle ^{\perp} \subset H^2(S,\BC). \] 
Since the monodromy group $\Image(\rho)$ is an arithmetic subgroup it is Zariski dense in $\SO(T)$\footnote{For $\langle \beta,\beta\rangle=0$ we consider $\{g\in\SO(H^2(S,\BC)\,|\,g(\beta)=\beta\}$ instead.}by Borel's Theorem~\cite{Bo66}:
\[\overline{\Image(\rho)}=\SO(T).\]
For clarity, let us comment on the relation between the monodromy and the Hodge group for general $S$ with 
\[\Pic(S) = \BZ \cdot \beta \] generated by an ample class $\beta$\footnote{Here and in Section~\ref{sec:Invariant theory} we consider the case $\langle \beta,\beta\rangle>0$ for simplicity.}. Since $T$ is an irreducible lattice of K3 type, the endomorphism ring 
\[K=\End_\BQ(T)\]
associated to $S$ is a field. For general $(S,\beta)$, we have $K=\BQ$. The computation by Zarhin~\cite{Za83} yields that the Hodge group of $S$ is the group of $K$-linear special isometries:
\[\Hdg(S)=\SO_K(T)\subset \SO(T)\, . \]
Thus, for general $S$ the Hodge group is the full special orthogonal group and equals the algebraic monodromy group. 

\section{Invariant theory}\label{sec:Invariant theory}
\subsection{Classical invariants}\label{subsec:Classical invariants}
We explain how the monodromy is used to prove Theorem~\ref{thm:1}.
Let $(S,\beta)$ be as above and write \[H=H^*(S,\BC) \] for the full cohomology. Consider the orthogonal decomposition \[H=H_{alg} \oplus T, \]
where $H_{alg}$ is spanned by $\{\one\, ,\beta\, ,\pt\}$ and 
\[T= \langle \beta \rangle ^{\perp} \subset H^2(S,\BC). \] 
We consider the action of the monodromy group $\SO(T)$ on the full cohomology $H$ by acting trivially on the algebraic part $H_{alg}$. Taking products induces a monodromy action on $H^*(S^n)$ for all $n$. We will require a description of the monodromy invariant part
\[ (H^{\otimes n})^{\SO(T)} \, .\]

The invariants are given by the trivial and the sign representations of the orthogonal group:
\[(H^{\otimes n})^{\SO(T)} = (H^{\otimes n})^{\Orth(T)}\oplus (H^{\otimes n})^{\Orth(T),\det}, \]
where 
\[ (H^{\otimes n})^{\Orth(T),\det}=\{v\in H^{\otimes n} \, | \, \forall \varphi\in \Orth(T): \varphi(v) = \det(\varphi) \cdot v\}. \]
Since $\SO(T)$ acts trivially on all powers $H_{alg}^{\otimes k}$, it suffices to determine the invariants of $\SO(T)$ acting on $T^{\otimes n}$. The answer is classically known \cite{Ho89,We39}. Let $\{\gamma_i\}$ be a basis of $T$ with intersection matrix $(g_{ij})$ and inverse $(g^{ij})$. The intersection pairing on $T$ corresponds to the class
\[ Q = \sum g^{ij} \gamma_i\otimes \gamma_j \in T^{\otimes 2}.\]
Consider the top wedge product of $T$ represented by some non-zero class $\Lambda$: 
\[\BC \cdot \Lambda=\Lambda^{21}T \subset H^{\otimes 42}\, ,\]
which is clearly a sign representation for $\Orth(T)$.
Then the invariants can be described as follows. The symmetric group $\Sym_n$ acts on $T^{\otimes n}$ by permutation and for $v\in T^{\otimes n}$ we denote by
\[ \BC \cdot \Sym_n(v) \] the linear subspace generated by all translates of $v$. Then\footnote{As above, the superscript `$\det$' indicates the sign representations.}
\[(T^{\otimes n} )^{\SO(T)}=(T^{\otimes n} )^{\Orth(T)} \oplus (T^{\otimes n} )^{\Orth(T),\det} \, ,\]
where
\[(T^{\otimes n} )^{\Orth(T)}= \begin{cases}
\BC \cdot \Sym_n(Q^{\frac{n}{2}}), &n\text{ even} \\
0,& n\text{ odd}
\end{cases} \]
and
\[(T^{\otimes n} )^{\Orth(T),\det}= \begin{cases}
\BC \cdot \Sym_n(\Lambda \otimes Q	^{\frac{n-21}{2}}), &n\geq 21\text{ odd} \\
0,&\text{else.}
\end{cases} \]
Note that $Q$ relates to the class of the diagonal of $S$ via
\[\Delta=Q + \one\otimes \pt + \frac{1}{\langle\beta,\beta\rangle}\beta\otimes \beta + \pt\otimes \one. \]
We arrive at a description of $(H^{\otimes n})^{\SO(T)}$. We have
\[(H^{\otimes n})^{\Orth(T)}=\BC \cdot \Sym_n( v_1 \otimes \ldots \otimes v_k\otimes \Delta^{\frac{n-k}{2}}), \]
where $v_i\in H_{alg}$ and $n-k$ even. Similarly, we find
\[ (H^{\otimes n})^{\Orth(T),\det}=\begin{cases}
0, &n<21 \\
\BC \cdot \Sym_n(\Lambda \otimes  w),& n\geq 21\end{cases} \]
where $w\in (H^{\otimes (n-21)})^{\Orth(T)}$. We will denote the former by\footnote{This is the subring studied in~\cite{Vo08}.}
\[ R^*(S^n)=(H^{\otimes n})^{\Orth(T)}. \] 
 \subsection{Vanishing}\label{subsec:Vanishing}
 Let $\alpha\in R^*(\Mbar_{g,n})$ be a tautological class. We apply the above monodromy argument to the class 
 \[ \pi^*\alpha \cap [\Mbar_{g,n}(S,\beta)]^{red}\]
 and find that \begin{equation} \ev_*\left( \pi^*\alpha \cap [\Mbar_{g,n}(S,\beta)]^{red}\right) = A + B\, ,\label{eq:A+B}\end{equation}
 where \[ A \in R^*(S^n)\, , \qquad B\in (H^{\otimes n})^{\Orth(T),\det}\, . \]
 We will show that $B=0$. Fix an orthonormal basis $\{\gamma_1,\ldots,\gamma_{21}\}$ of $T$ and let \[\gamma=\gamma_1\, .\]
 The class $\Lambda$ (see Section~\ref{subsec:Classical invariants}) can then be described as
 \[ \Lambda = \sum_{\sigma \in \Sym_n} (-1)^{\sigma} \sigma (\gamma_1\otimes\ldots \otimes \gamma_{21})\, . \] 
 Since 
 \[ \langle \Lambda, \Lambda \rangle \neq 0\,  \]
 we can find $B^{\vee}\in (H^{\otimes n})^{\Orth(T),\det}$, such that
 \[ \langle A, B^{\vee} \rangle = 0, \qquad \langle B, B^{\vee} \rangle = 1\, ,\]
 where $\langle \, , \, \rangle$ denotes the intersection pairing on $S^n$. We deduce Theorem~\ref{thm:1} from the following Lemma by pairing equation~\eqref{eq:A+B} against $B^{\vee}$. For this, let $v_i\in H^*(S)$ such that $\langle \gamma, v_i \rangle = 0$ for all $i$. We will use a bracket notation for Gromov--Witten invariants:
 \[ \left\langle \alpha; \tau_0(\gamma)^k \prod \tau_0(v_i) \right\rangle_{g,n,\beta}^S = \int_ {[\Mbar_{g,n}(S,\beta)]^{red}}\pi^*\alpha \cup \ev_1^*\gamma\ldots \ev_k^*\gamma \prod \ev_i^*v_i\, . \]
 
 \begin{lemma}\label{lemma:Vanishing} Assume that $\beta\in H_2(S,\BZ)$ is primitive and $k$ is odd. Then we have the vanishing \[ \left\langle \alpha; \tau_0(\gamma)^k \prod \tau_0(v_i) \right\rangle_{g,n,\beta}^S= 0 \, .\]
 \end{lemma}
 \begin{proof}[Proof of Theorem~\ref{thm:1}]
	 Let $B^{\vee}$ be as above. By the description of the invariants in Section~\ref{subsec:Classical invariants} we see that $B^{\vee}$ is a sum of ($\Sym_n$ translates of) classes
	 \[ \Lambda \otimes  w\, ,\]
	 with $w\in (H^{\otimes (n-21)})^{\Orth(T)}$. The class $\Lambda$ contributes one appearance of $\gamma$ for each summand. By the description of $(H^{\otimes m})^{\Orth(T)}$, each generator contributes an even number of $\gamma$'s coming from the diagonal classes. Thus, $B^{\vee}$ is a sum of classes each of which has on odd number of $\gamma$'s
 \end{proof}
 \begin{proof}[Proof of Lemma~\ref{lemma:Vanishing}]
 	The proof follows the induction in~\cite[Sec.\ 7]{MPT10} (see also~\cite[Sec.\ 3]{OP18}). We use the lexicographic order on $(g,n)$. A key ingredient is the strong form of Getzler--Ionel vanishing proved in~\cite{FP05}. Recall that for a stable graph $\Gamma$ and the corresponding boundary map 
 	\[ \iota\colon \Mbar_{\Gamma} \to \Mbar_{g,n} \]
 	we have the following splitting behavior for the reduced class. The pullback 
 	\[\iota^! [\Mbar_{g,n}(S,\beta)]^{red}\] decomposes as a sum of Gysin pullbacks of products of virtual classes. Exactly one vertex attains a reduced virtual class and all other vertices attain the usual virtual class. Since the virtual class is non-trivial only for $\beta=0$ and $g\in\{0,1\}$, very few stable graphs $\Gamma$ contribute to the splitting formula. Also, recall that in this case 
	\[ \Mbar_{g,n}(S,0) \cong \Mbar_{g,n} \times S \] 
	and the virtual class is given by
		\[[\Mbar_{g,n}(S,0)]^{vir} = \begin{cases} [\Mbar_{0,n} \times S], &g=0 \\
	\pr_2^*c_2(S) \cap [\Mbar_{1,n} \times S], & g=1 \\
	0, & g\geq 2\, . \end{cases} \]
	We start with the case of genus zero invariants:
	
	\noindent{\bf Base.} Let $(g,n,\beta)=(0,n,0)$ and consider the non-reduced invariant
 	\[\int_ {[\Mbar_{0,n}(S,0)]^{vir}}\pi^*\alpha \cup \ev_1^*\gamma\ldots \ev_k^*\gamma \prod \ev_i^*v_i\, = \deg(\alpha) \cdot \int_S \gamma^k \prod v_i\, . \]
 	The integral vanishes since $k$ is odd.
 	
	For $g=0$, $\alpha\in R^0(\Mbar_{g,n})$ and $\beta\neq 0$ we use the divisor equation. The corresponding reduced invariant vanishes since $\langle \beta,\gamma\rangle=0$.
 	
 	\noindent{\bf Induction.} We distinguish two cases depending on the codimension of $\alpha\in RH^*(\Mbar_{g,n})$. If no point insertion appears, i.e.\ $v_i\in H^{\leq 2}(S)$ for all $i$, then $\alpha\in RH^{\geq g}(\Mbar_{g,n})$ and $\alpha$ is supported on the boundary 
 	\[ \partial \Mbar_{g,n} \subset \Mbar_{g,n}\, . \]
 	There are only three types of boundary stata with possibly non-zero contributions.
 	
 	\noindent{\bf Boundary 1.} Consider the boundary map gluing the last two markings
 	\[\iota\colon \Mbar_{g-1,n+2} \to \Mbar_{g,n}\, . \]
 	If $\alpha = \iota_*\alpha'$, then
 	\[\left\langle \alpha; \tau_0(\gamma)^k \prod \tau_0(v_i) \right\rangle_{g,n,\beta}^S= \left\langle \alpha'; \tau_0(\gamma)^k \prod \tau_0(v_i) \tau_0(\Delta) \right\rangle_{g-1,n+2,\beta}^S= 0 \] 
 	by induction. Indeed, the diagonal class contributes two insertions of $\gamma$, i.e.\ the number of $\gamma$ insertions is again odd.
 	
 	\noindent{\bf Boundary 2.} Let $\ell\geq 1$ and $k=k_1+\ldots + k_{\ell}$. Consider the boundary map
 	\[\iota \colon \Mbar_{g-\ell, n-k+\ell} \times \Mbar_{1,k_1+1}\times \ldots  \times \Mbar_{1,k_{\ell}+1}\to \Mbar_{g,n} \, ,\]
 	where the markings specified by $I_j\subset\{1,\ldots,n\}$ lie on the $j$-th genus 1 component. Let $I= I_1 \cup \ldots \cup I_{\ell}$. Note that each gluing induces an appearance of the diagonal class $\Delta$. Recall the formula for the diagonal \[ \Delta = \one \otimes \pt + \frac{1}{\beta^2} \beta \otimes \beta +\sum \gamma_i\otimes \gamma_i + \pt\otimes\one\, . \]
 	Only the last summand produces possibly non-zero invariants at the genus 1 vertices.
 	Thus, if $\alpha= \iota_* (\alpha '\times\alpha_1\times\ldots\times \alpha_{\ell})$ then 
 	\begin{align*} \left\langle \alpha; \tau_0(\gamma)^k \prod \tau_0(v_i) \right\rangle_{g,n,\beta}^S = &\left\langle \alpha'; \tau_0(\gamma)^k \prod_{i\notin I} \tau_0(v_i)\tau_0(\pt)^\ell \right\rangle_{g-\ell,n-k+\ell,\beta}^S \\
 		&\cdot \prod_{j=1}^\ell \left( \deg \alpha_j\cdot  \int_S c_2(S) v_{I_j} \right)\, . \end{align*}
 	 The integral $\int_S c_2(S) v_{I_j}$ is non-zero only if all the insertions specified by $I_j$ are multiples of $\one\in H^*(S)$. Thus, all markings with $\gamma$ insertions lie on the non-contracted component and we can apply induction to find
 	\[ \left\langle \alpha; \tau_0(\gamma)^k \prod \tau_0(v_i) \right\rangle_{g,n,\beta}^S = 0\, .\]
 	
 	\noindent{\bf Boundary 3.} Consider the case of a genus $g$ component with $\ell$ contracted bubbles attached, i.e.\ 
 	\[ \iota\colon \Mbar_{g,n-k+\ell}\times \Mbar_{0,k_1+1}\times \ldots  \times \Mbar_{0,k_{\ell}+1}\to \Mbar_{g,n} \, .\]
 	We use the same notation as in the previous case. This time, the invariants at the genus 0 vertices are of the form
 	\[ \deg \alpha_i \cdot \int_S \delta v_{I_j} \, .\]
 	Here, $\delta$ is part of the basis $\{\one,\beta,\gamma_1,\ldots,\gamma_{21},\pt\}$ of $H^*(S)$ and comes from the appearance of the diagonal class at the gluing points. If $I_j$ does not contain any marking carrying a $\gamma$ insertion, the integral is non-zero only for $\delta\neq \gamma$ and we can apply induction (lower $n$). However, if $I_j$ contains a marking carrying a $\gamma$ insertion, there are two cases.
 	\begin{enumerate}
 		\item If $I_j$ contains no other such marking, then the integral is non-zero only for $\delta=\gamma$ in which case we obtain a new appearance of $\gamma$ at the genus $g$ vertex. The number of $\gamma$'s at the genus $g$ vertex remains unchanged.
 		\item If $I_j$ contains more than one such marking, then it contains exactly two and the integral is non-zero only for $\delta=\one$. In this case, the number of $\gamma$'s at the genus $g$ vertex is reduced by two and therefore stays odd.
 	\end{enumerate}
 	
 In each case the number of markings at the genus $g$ vertex decreases and we can apply induction.
 
 Finally, to deal with point insertions we work with an elliptic K3 surface $S$ with a section and use degeneration to the normal cone of an elliptic fiber $E$, i.e.\
 \[ S \rightsquigarrow S \cup_E E\times \BP^1\,  \]
 
 For this, we recall the relation between absolute invariants of $S$ and relative invariants of $(S,E)$.
  
 \noindent{\bf Upper triangularity.} Consider an elliptic K3 surface $S$ with a section. Degeneration to the normal cone of an elliptic fiber $E$ shows that the Gromov--Witten invariants of the relative geometry $(S,E)$ are determined by the absolute invariants. More precisely, the degeneration formula yields an upper triangular relation (see~\cite[Lem.\ 31]{MPT10}, \cite[Prop.\ 4]{OP18}).
 
 \noindent{\bf Point insertions.} We consider the case of point insertions. This time, we can not apply Getzler--Ionel vanishing. Instead, we use the degeneration
 \[ S \rightsquigarrow S \cup_E E\times \BP^1 \, . \]
and  specialize the point to lie on the bubble $E\times \BP^1$. The degeneration formula for the reduced class~\cite{LL05,LL06,MPT10} then removes the point insertion from
\[\left\langle \alpha; \tau_0(\gamma)^k \prod \tau_0(v_i) \tau_0(\pt) \right\rangle_{g,n,\beta}^S\, . \]
In fact, if $g= g' + g''$ is a splitting of the genus, then the corresponding summands in the degeneration formula are non-zero only for $g''\geq 1$, since the relative geometry $E\times \BP^1/E$ carries a relative and a non-relative point insertion. Thus, the genus $g'$ drops\footnote{We crucially use primitivity of $\beta$ at this point. A proposal for a degeneration formula in the imprimitive case has been made in~\cite[Sec.\ 4.6]{MPT10}. However, the genus reduction for imprimitive curve classes will be more subtle.} and we conclude the proof using the upper triangular relation between relative and absolute invariants.\end{proof}
 
 We deduce the invariance property from Theorem~\ref{thm:1}. 
 \begin{proof}[Proof of Corollary~\ref{cor:Inv}]
 Let \[\varphi:\Big(H^2(S,\BR)\, ,\langle,\rangle\Big) 
 \rightarrow \Big(H^2(\widetilde{S}, \BR),\langle,\rangle \Big) \]
 be a real isometry with
 \[\varphi(\beta)= \widetilde{\beta}\, .\]
 After extending $\varphi$ to 
 \[ \varphi \colon H^*(S^n) \to H^*(\widetilde{S}^n) \] 
 we see that $\varphi$ respects the subrings $R^*(S^n)$ resp.\ $R^*(\widetilde{S}^n)$. Applying Theorem~\ref{thm:1} we have
 \[ \varphi\ev_* \left(\prod \psi^{a_i} \cap [\Mbar_{g,n}(S,\beta)]^{red}\right) = \ev_* \left(\prod \psi^{a_i}\cap \ev_* [\Mbar_{g,n}(\widetilde{S},\widetilde{\beta})]^{red}\right)\, .\]
We conclude by integrating against the insertions $w_i$ and using once again that $\varphi$ is an isometry.
\end{proof}

\section{Classical techniques}\label{sec:Classical techniques}
This section contains a proof of Theorem~\ref{thm:2}. We use degeneration formulae and virtual localization. The Gromov--Witten classes for relative geometries are defined analogously to the absolute case. For stable maps to rubber~\cite{MP06} the corresponding GW classes are indicated by a tilde.

Let $S$ be a nonsingular projective K3 surface, $\beta\in H_2(S,\BZ)$ effective with $\langle\beta,\beta\rangle = 2h-2\geq0$ and $v\in R^*(S^n)$. By standard arguments we may choose $S$ to be an elliptic K3 surface with a section. 

We prove Theorem~\ref{thm:2} in the following five steps:
\begin{enumerate}
	\item Degenerate $S$ to a broken geometry 
	\[ S \rightsquigarrow R\cup_{E} R\to \BP^1\cup\BP^1\, ,\] with $R\cong \Bl_9(\BP^2)$ a rational elliptic surface and $E\subset R$ a smooth elliptic fiber. Since $\beta$ and all diagonals of $S^n$ lift to the family, we can apply the degeneration formula to $v\in R^*(S^n)$. We are reduced to proving that the classes $\GW{R/E}{g,n,\beta}(v|\one)$ of the pair $(R,E)$ are tautological.
	\item Degeneration to the normal cone 
	\[ R \rightsquigarrow R\cup_E E\times \BP^1 \]
	reduces to $R$ and the pair $(E\times \BP^1, E)$. Since $R$ can be deformed to a toric blow-up, the GW classes are tautological by virtual localization. It remains to consider classes of the form $\GW{E\times \BP^1/ E}{g,n,\beta}(v|\eta)$ with $\eta\in H^*(E)$.
	\item Virtual localization for the pair $(E\times \BP^1, E)$ reduces to $E$ and rubber GW classes.
	\item The rubber case involves $\Psi$-classes defined by the cotangent line at the relative divisor. Topological recursion relation removes the $\Psi$-classes. 
	\item The rubber GW classes are tautological by the recent work on DR cycles for target varieties and the elliptic curve case.
\end{enumerate}

\begin{proof}[Proof of Theorem~\ref{thm:2}]
		Consider an elliptic K3 surface $S \to \BP^1$ with a section. Denote the class of the section and the class of a fiber by $B$ resp.\ $F$ and let
		\[ \beta_h = B + h F \in H_2(S,\BZ). \]
		We will use the analogous notation for any of the surfaces $S$, $R$ and $E\times \BP^1$.
		
		We treat the non-positive case $\langle\beta,\beta\rangle=-2$ first. Then, the (cycle-valued) Gromov--Witten theory is determined by the normal bundle $\CN_{B/S}\cong \CO_{\BP^1}(-2)$. By virtual localization~\cite{GP99} all Gromov--Witten classes are tautological. From now on assume $h\geq 1$.

		\vspace{8pt}
		\noindent {\bf Step 1.} Our starting point is a degeneration of $S$ to an elliptically fibered surface $R\cup_E R$ (see \cite[Sec.\ 2.2, 2.3]{MPT10}), where $R\cong \Bl_9(\BP^2) \to \BP^1$ is the blow-up of a general elliptic pencil. 
		
	\begin{figure*}[h]
		\begin{center}
			\includegraphics{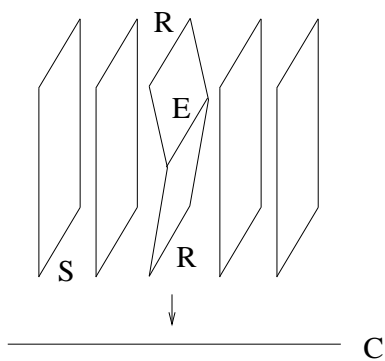} \vspace{-4mm}
			\caption{A degeneration of $K3$ surfaces, see~\cite[Figure 1]{MPT10}. \label{fig}}
		\end{center}
	\end{figure*}

	Degeneration formula~\cite{MPT10} for the reduced virtual class of $S$  expresses GW classes of $S$ in terms of classes of the pair $(R,E)$ with relative insertion $\one\in H^*(E)$.  
	More precisely, let $\Mbar_{g,n}(R/E,\beta)$ be a moduli space of stable relative maps to expanded degenerations of $(R,E)$ with multiplicity $1$ along the divisor $E$ and define the Gromov--Witten classes of $(R,E)$ using this moduli space. Let
	\[ \iota\colon \Mbar_{g',n'+1} \times \Mbar_{g'',n''+1} \to \Mbar_{g,n} \]
	be a gluing map. We have\footnote{Recall that we use the reduced virtual class for $S$. Also, we use the same notation $\iota$ for all gluing maps.}
		\[\GW{S}{g,n,\beta_h}(v) = \sum \iota_*\left( \GW{R/E}{g',n',\beta_{h'}}(v_I|1) \times \GW{R/E}{g'',n'',\beta_{h''}}(v_J|1)\right)\, , \]
	where the sum is over all
	\[ I\sqcup J = \{1,\ldots,n\}\, ,\quad  g' + g'' = g\, ,\quad h' + h'' = h\, . \]

	\vspace{8pt}
		\noindent{\bf Step 2.} Consider the degeneration to the normal cone of $E\subset R$
	\[ R \rightsquigarrow R\cup_E (E\times \BP^1).\]
		Since $B$ is rigid in $R$, the evaluation map 
		\[ \Mbar_{g,n}(R/E,\beta) \to E \]
		is constant. 
		We fix a self-dual basis $\{\eta\}$ of $H^*(E)$ given by $\{\one, \sfa, \sfb, \pt\}$. The degeneration formula yields:
		\[\GW{R}{g,n,\beta}(v)= \GW{R/E}{g,n,\beta}(v|\one)\times \GW{E\times \BP^1/E}{0,0,B}(\emptyset|\pt) + \ldots\, . \]
		By `$\ldots$' we mean products of GW classes involving $\GW{R/E}{g',n',\beta'}(v'|\one)$  with
		
		\begin{enumerate}
			\item $g'<g$ or
			\item $g'=g$ and $n'<n$ or
			\item $g'=g$ and $n'=n$ and $\beta - \beta'\neq 0$ is effective.
		\end{enumerate}
		We are thus reduced to the case of GW classes of $R$ and $(E\times \BP^1,E)$. \medskip
		
		Note that $R$ is deformation equivalent to a toric surface. Indeed, $R$ is isomorphic to the blow-up $\Bl_9(\BP^2)$ of $\BP^2$ in the base locus of a general elliptic pencil, i.e.\ in $9$ points. Thus, $R$ can be viewed as a blow-up $\Bl_p(R')$ and by induction we may assume that $R'$ is deformation equivalent to a toric surface. Moving the point $p$ to a fixed point of the torus action deforms $\Bl_p(R')$ to a toric blow-up of $R'$. 
		Virtual localization and deformation invariance of the virtual fundamental class imply that all GW classes of $R$ are tautological (see~\cite[p.\ 20]{GP99}).
		
		\vspace{8pt}
		\noindent{\bf Step 3.} Consider the $\BC^*$-action on $E\times \BP^1$ defined by the fiberwise action on $\BP^1$ with weights $(-1,0)$. Let $p_0=[1:0]$ and $p_{\infty}=[0:1]$ be the fixed points. The induced action on the tangent spaces $T_{p_0}$ and $T_{p_{\infty}}$ has weight $1$ resp.\ $-1$. We denote the equivariant parameter by~$t$. We use virtual localization to prove that the GW classes of the pair $(E\times \BP^1,E_0)$ are tautological. \medskip
			
		We first describe the fixed loci of the induced $\BC^*$-action on 
		\[\Mbar_{g,n}(E\times \BP^1/E_0, \beta)\, ,\]
		the moduli space of stable maps to expanded degenerations of $(E\times \BP^1, E_0)$ with mulitplicity $1$ along the relative divisor. The domain of a $\BC^*$-fixed stable map is a gluing 
			\[ C = C'\cup C_e \cup C'' \]
			of a unique rational compononent $C_e$ mapping isomorphically to $B$, with various curves attached over $p_0$ resp.\ $p_{\infty}$:
			\begin{itemize}
				\item $C'$ defines a stable map to rubber,\footnote{See e.g.\ \cite[Sec.\ 1.5]{MP06} for details about rubber targets.}
				\item $C''$ defines a stable map to $E_{\infty}$.
			\end{itemize} We let $x_{\infty}\in C'$ and $x_0\in C''$ be the gluing points.
			Fixing splitting data 
			\[\Gamma=(g',h',I) \]
			with $0\leq g' \leq g$, $0\leq h'\leq h$, $I \subset \{1,\ldots,n\}$
			corresponds uniquely to the fixed locus $\Mbar_{\Gamma}$ of stable maps with:
			\begin{itemize}
				\item $g(C')=g'$,
				\item $f_*([C'])=B+h'F$,
				\item the markings specified by $I \subset \{1,\ldots,n\}$ lie on $C'$.
			\end{itemize}
		
		Following~\cite{GP99} we use the standard tangent-obstruction sequence to compute the moving part of the obstruction bundle. We let $\BE^{\vee}$ be the dual of the Hodge bundle associated to $C''$ and denote by \[ c_z(F) = 1+z c_1(F) + \ldots + z^r c_r(F) \]
		the Chern polynomial of a bundle $F$ of rank $r$. Also, let $\psi_{\infty}$ be the cotangent line class at the relative marking $x_{\infty}$ and $\psi_0$ the cotangent line class at $x_0$. Then
		\[ \frac{1}{e(N_{\Gamma}^{vir})} = \frac{ (-t)^{g''-1} {c_{-1/t}(\BE ^{\vee})}}{(t-\psi_{\infty} ) (-t - \psi_0) }\, ,\]

		Let $v_1,\ldots, v_n \in H^*(E\times \BP^1)$, $\delta\in H^*(E)$ and set $v=\prod \ev_i^*(v_i)$. We denote the insertions specified by the splitting
		\[ \{1,\ldots, n \} = I \sqcup J \] 
		by $v_I$ and $v_J$ respectively. We choose equivariant lifts $\widetilde{\delta}$ of $\delta$ and $\widetilde{v_I}$ of $(v_I)_{|E_0}$ and let $\xi(\Gamma)$ be the usual combinatorial factor. Then
	\begin{align*} \GW{E\times \BP^1/E}{g,n,\beta}(v|\delta) = \mathlarger{\sum}_{\Gamma} \frac{\xi(\Gamma)}{e(N_{\Gamma}^{vir})} 
	 & \iota_*\Big( \sum_{\eta}  \GW{E\times \BP^1\sim}{g_1,n_1,\beta_{h_1}}(\widetilde{\delta}|\widetilde{v_I}|\eta)  \\
	& \times\GW{E}{g_2,n_2+1,h_2}(\eta^{\vee}(v_J)_{|E_{\infty}})\Big) \, .
	\end{align*}
		
		\vspace{8pt}
		\noindent{\bf Step 4.} Consider a moduli space of stable maps to rubber \[\Mbar^{\sim}=\Mbar_{g,n} (E\times\BP^1/(E_{0}\cup E_{\infty}),\beta)^{\sim}.\]
		Note that the expression in Step~3 involves powers of the class
		\[\psi_{\infty} \in H^2(\Mbar^{\sim},\BZ)\, . \]
		We describe briefly how to remove them (cf.\ \cite[Sec.\ 1.5]{MP06}). Let $[f]$ be a moduli point and assume that the domain of $f$ carries at least one non-relative marked point~$q$. Let $\CM_{0,3}$ be the stack of genus $0$ prestable curves with three marked points. Consider the map
		\begin{align*} \Mbar^{\sim} &\to \CM_{0,3} \\ [f] &\mapsto\pr^{-1} \pr(q)\, , \end{align*}
		 where $\pr$ is the projection from expanded degenerations to $E\times \BP^1$.
		 The contangent line class at the relative divisor $\psi_{\infty}$ is the pull-back of the cotangent line class at the third marking. Topological recursion relation yields
		 \begin{align*}\GW{E\times \BP^1\sim}{g,n,\beta}(\widetilde{\delta}|\widetilde{v}\psi_{\infty}^k |\eta) = \sum \zeta(\nu)  &\iota_* \Big(\GW{E\times \BP^1\sim}{g',n',\beta'}(\widetilde{\delta}|\widetilde{v'}|\nu)\\
		 &\times \GW{E\times  \BP^1\sim}{g'',n'',\beta''}(\nu^{\vee}|\widetilde{v''}\psi_{\infty}^{k-1}|\eta) \Big)\, ,\end{align*}
		where $\zeta(\nu)$ is a combinatorial factor. We thus remove all powers of $\psi_{\infty}$-classes recursively. If the domain of $f$ does not carry a non-relative marking, we use the dilaton equation to add a marking $q$:
		 \[\GW{E\times \BP^1\sim}{g,n,\beta}(\widetilde{\delta}|\widetilde{v}\psi_{\infty}^k |\eta) = \frac{1}{2g+n}\GW{E\times \BP^1\sim}{g,n+1,\beta}(\widetilde{\delta}|\widetilde{v}\psi_q\psi_{\infty}^k |\eta)\, .\]
		 
		\vspace{8pt}
		\noindent{\bf Step 5.} Consider a rubber Gromov--Witten class $\GW{E\times \BP^1\sim}{g,n,\beta}(\widetilde{\delta}|\widetilde{v}|\eta)$ without $\psi_{\infty}$-classes. Let \[\epsilon\colon \Mbar^{\sim} \to \Mbar_{g,n+2}(E,h) \]
		be the natural forgetful morphism. We have
		\[ \GW{E\times \BP^1\sim}{g,n,\beta}(\widetilde{\delta}|\widetilde{v}|\eta)= \pi_* \left( \ev^*(\widetilde{\delta}\widetilde{v}\eta) \cap \epsilon_*[\Mbar^{\sim}]^{vir} \right)\]
		The push-forward $\epsilon_*[\Mbar^{\sim}]^{vir}$ of the virtual class is the $E$-valued double ramification cycle~\cite{JPPZ18}; it is a tautological class on $\Mbar_{g,n+2}(E,h)$. The Gromov--Witten classes of $E$ are tautological by~\cite{Ja17} and the proof is complete.
		\end{proof}

\providecommand{\bysame}{\leavevmode\hbox to3em{\hrulefill}\thinspace}
\providecommand{\MR}{\relax\ifhmode\unskip\space\fi MR }
\providecommand{\MRhref}[2]{%
	\href{http://www.ams.org/mathscinet-getitem?mr=#1}{#2}
}
\providecommand{\href}[2]{#2}

	\end{document}